\documentclass{amsart}
\usepackage[utf8]{inputenc}
\usepackage{latexsym, amssymb, amsmath, amscd, amsthm, amsxtra}
\usepackage{amssymb}
\usepackage{mathrsfs}
\usepackage{collectbox}



\newtheorem{thm}{Theorem}[section]

\newtheorem{lem}{Lemma}[section]
\newtheorem{cor}{Corollary}[section]

\DeclareMathOperator{\di}{div}

\begin{document}
\title{Carleman Estimate for Surface in Euclidean Space at Infinity}
\author{Ao Sun}
\email{aosun@mit.edu}
\address{Massachusetts Institute of Technology\\
Department of Mathematics}
\date{\today}
\maketitle
\begin{abstract}
This paper develops a Carleman type estimate for immersed surface in Euclidean space at infinity. With this estimate, we obtain an unique continuation property for harmonic functions on immersed surfaces vanishing at infinity, which leads to rigidity results in geometry.
\end{abstract}

\section{Introduction}

Let $f$ be a harmonic function defined in a connected open set $U\subset\mathbb{R}^n$. If $f\equiv0$ in an open subset of $U$, then $f\equiv0$ in $U$ because it is analytic. More generally, if derivatives of all orders of $u$ vanish at some point $p\in U$, then $f$ is constant in $U$. This property is called the {\bf unique continuation} property for harmonic functions. In general, unique continuation property refers the phenomenon that if the solution to certain differential equations vanishes to $\infty$ order somewhere, then it must vanish identically.

Generally, there are two major approaches to show unique continuation for solutions to elliptic equations. One uses Carleman estimates, which is a weighted-$L^2$ inequality first developed by Carleman. See  ~\cite{AKS}, ~\cite{P} for some early results, and  ~\cite{JK}, ~\cite{K}, ~\cite{KT} for some recent developments. The other approach uses the Almgren frequency function, which is first developed by Garofalo and Lin in ~\cite{GL}. Also see ~\cite{CM2} for the application of the
second approach.

In this paper, we prove Carleman estimates for immersed surface and get the unique continuation property for functions on immersed surface in Euclidean space. As a result, we get the following rigidity theorem:

\begin{thm}[A part of Theorem \ref{thmplane}]
Suppose $\Sigma$ is an immersed minimal surface in $\mathbb{R}^n$ with at most exponential area growth. If one end of $\Sigma$ is asymptotic to a plane exponentially, in the sense that this end can be viewed as a graph of a vector valued function $u$ over the plane, such that
$$\limsup_{r=\vert x\vert\to\infty}\vert e^{\sigma r}u(x)\vert=0,\mbox{ for all $\sigma>0$}$$

then $\Sigma$ is just the plane.
\end{thm}

Another rigidity result in this paper concerns the minimal surfaces in $\mathbb{R}^3$

\begin{thm}[Theorem \ref{thmr}]
Suppose $\Sigma$ is a minimal surface in $\mathbb{R}^3\setminus B_{R_0}$ with at most exponential area growth and $\vert x^\top\vert^2\geq (1-\lambda)r^2$ for some $\lambda<1$. Moreover we assume $\vert A\vert\leq C/r^{1/4}$ for some constant $C$. Then if $\Sigma'$ is a minimal surface, which asymptotic to a same end with $\Sigma$ in the following sense: $\Sigma'$ can be viewed as a graph $\Sigma_u$ over $\Sigma$, and for all $\sigma >0$ $$\limsup_{r\to\infty}\vert e^{\sigma r}u\vert=0,\vert\nabla u\vert \leq C\frac{\vert u\vert^{1/2}}{r^{1/4}}$$
Then $\Sigma'$ must coincide with $\Sigma$.
\end{thm}

Above theorems are generalizations of unique continuation to minimal surfaces: if an end of a minimal surface is asymptotic to an end of another minimal surface in certain way, then they must be coincident. These theorems are just some of the applications of the following theorem of unique continuation:

\begin{thm}
\label{thmu}
Let $\Sigma=\Sigma^2$ be an immersed surface in $\mathbb{R}^n$. Suppose there are constants $R_0,\lambda,C_1,C_2,t$ such that over the part of $\Sigma$ outside $B_{R_0}$,

\begin{enumerate}
    \item $\vert A\cdot x\vert\leq C_1<1/15$ and $2+H\cdot x>0$
    \item $\vert\nabla H\cdot x\vert\leq C_2$
    \item $\Sigma$ tilts no more than $\lambda<1$, i.e. for any $x\in\Sigma$, $\vert x^\top\vert^2\geq (1-\lambda)r^2$
    \item area growth of $\Sigma$ is at most exponential, i.e. $Area(\Sigma\cap B_r)\leq e^{tr}\mbox{ for any $r>R_0$}$
\end{enumerate}

Then for $u$ a function on $\Sigma\setminus B_{R_0}$ satisfying the following elliptic inequality:
\begin{equation}
\vert\Delta_\Sigma u\vert\leq\frac{C}{r^{1/2}}\vert u\vert,\mbox{ for some constant $C$}
\end{equation}
with exponential decay to $0$ at infinity, i.e.
\begin{equation}
\limsup_{r\to\infty}\vert e^{sr}u\vert=0 \mbox{ for any fixed $s>0$}
\end{equation}
Then $u\equiv0$
\end{thm}

Certain results for Euclidean spaces have been obtained in previous researches by Carleman estimates. For example, Koch and Tataru completely list all results for the following question in ~\cite{KT}: for what $V,W_1,W_2$, if the solution $u$ to divergent form elliptic equation

$$\partial_i(g^{ij}\partial_ju)=Vu+W_1\nabla u+\nabla(W_2u)$$

satisfies the decay at $x_0\in\mathbb{R}^n$:

$$\int_{B_r(x_0)}\vert u\vert^2\leq C_Nr^N$$

holds for all $r$ less than some fixed $R$ and any $N>0$, then $u$ is identically $0$ near $x_0$. Also see ~\cite{K} for many classical results.

The main issue for generalizing unique continuation to immersed surface is that the geometry of the immersion is involved in. Then some geometric quantities of the immersed surface would appear in the Carleman estimates we get. As a result, our estimates is applicable to those immersed surfaces with some geometric restraints.

\subsection{Applications} There are two main applications being discussed in this paper.

Firstly, the unique continuation property is valid for immersed surface with certain geometric property, which means the violation of the unique continuation property is an obstruction of such kind of immersion. In particular, we show that the surfaces without the property of unique continuation can not have an end that looks like a plane or cone. More precisely, we show the following theorem

\begin{thm}\label{app1}
Let $\Sigma$ be a $2$-dim non-compact surface with complete metric. If there exists a non-constant function $u$ defined on the part of ends of $\Sigma$ which satisfies the elliptic inequality

$$\vert\Delta_\Sigma u\vert\leq\frac{C}{d^{1/2}}\vert u\vert$$

with exponential decay on $\Sigma$, i.e.

$$\limsup_{d\to\infty}\vert e^{2\sigma d}u\vert\to0\mbox{ for any $\sigma>0$}$$

Then $\Sigma$ can not be isometrically immersed to $\mathbb{R}^n$ for any $n\geq3$, which ends are asymptotic to planes or cones in $C^3$ sense.
\end{thm}

Secondly, we show the rigidity of some special surfaces in Euclidean space, especially minimal surfaces. We have already presented two such results at the beginning of this section.

The idea of these rigidity results is that many special surfaces in Euclidean space have some geometric quantities which are harmonic functions or solutions to elliptic inequalities. See ~\cite{CM} for minimal surface case and ~\cite{CM3} for self shrinker case. The unique continuation of those geometric quantities would give us rigidity of the surface.

\subsection{Arrangement of the Paper}Here is the arrangement of this paper. In section 2, we develop the Carleman inequality for functions on surface in $\mathbb{R}^n$. In section 3, we use Carleman estimates which we obtained in section 2 to show the unique continuation property for functions on immersed surface. We can obtain the unique continuation property not only for harmonic functions, but also for the solutions to certain elliptic inequalities. Finally, we discuss the application of the unique continuation property we obtained in this paper in section 4 .

\subsection{Acknowledgement}

The author wants to thank Professor William Minicozzi and Professor David Jerison for their inspiring and helpful comments; and thanks Zhichao Wang and Jonathan Zhu for very helpful discussion with them.

\section{Carleman Inequality for Functions on Surface in $\mathbb{R}^n$}

Let $\Sigma=\Sigma^k$ be a $k$-dimensional immersed submanifold in $\mathbb{R}^n$. We will use $\nabla_\Sigma,\di_\Sigma,\Delta_\Sigma$ to denote the intrinsic differential operators on $\Sigma$. For $p\in\Sigma$ and any vector $V\in T_p\mathbb{R}^n$, we will use $V^\top$ and $V^\bot$ to denote the tangential part and the normal part of $V$ respectively.

$H$ and $A$ will be used to denote the mean curvature vector and the second fundamental forms on $\Sigma$. We will always use $e_1,\cdots,e_k$ to represent a local orthonormal basis over $\Sigma$. With such an orthonormal frame, we can write $H=\sum_{i=1}^k\nabla_{e_i}e_i$, and $A=(A_{ij})=(\nabla_{e_i}e_j)$. We will use $\vert A\cdot x\vert^2$ to denote $4\vert\sum_{i,j=1}^k\langle x,e_j\rangle\langle x,\nabla_{e_i}e_j\rangle\vert^2$

We will use $x$ to denote the position vector in $\mathbb{R}^n$, view it as a vector in $T_x\mathbb{R}^n$, and the distance function $r=r(x)$ is defined to be the Euclidean distance from $x$ to $0$. We have the following computations:

$$\nabla_\Sigma r=\frac{x^\top}{r}$$

$$\di_\Sigma x=k$$

$$\di_\Sigma x^\top=\di_\Sigma (x-x^\bot)=k-\di_\Sigma(x^\bot)=k+H\cdot x$$

$$\Delta_\Sigma r=\di_\Sigma (\frac{x^\top}{r})=\frac{k+H\cdot x}{r}-\frac{\vert x^\top\vert^2}{r^3}$$

$$\nabla_\Sigma\vert x^\top\vert^2=2x^\top+T$$

Here $T$ equals $\sum_{i,j=1}^k2\langle x,e_j\rangle\langle x,\nabla_{e_i}e_j\rangle e_i$ in local orthonormal basis. By direct computation we have

$$\nabla_\Sigma\vert x^\top\vert^2\cdot x=A(x^\top,x^\top)\cdot x$$



Now we state our Carleman estimates.

\begin{thm}
\label{thm1}
Let $\Sigma=\Sigma^2$ is a $2$-dimensional submanifold in $\mathbb{R}^n$, moreover suppose there are constants $R_0,\lambda,C_1,C_2$ such that over the part of $\Sigma$ outside $B_{R_0}$, condition (1),(2),(3) of Theorem \ref{thmu} hold.

Then for any smooth function $v$ compactly supported on $\Sigma\setminus B_{R_0}$, the following inequality

\begin{equation}
\int re^{2\sigma r}(\Delta_\Sigma v)^2\geq C\sigma^2\int e^{2\sigma r}v^2
\end{equation}

holds for $\sigma>\sigma_0$, where $\sigma_0$ only depends on $C_1,C_2,\lambda,R_0$. Here the integral is taken over $\Sigma$, $C=C(C_1,C_2,\lambda,R_0)$ is a constant independent of $\sigma$.
\end{thm}

\begin{proof}[Proof of Theorem \ref{thm1}]

Let $v=e^{-\sigma r}w$. We have the following equalities

$$\nabla_\Sigma v=e^{-\sigma r}(\nabla_\Sigma w-\sigma w\nabla_\Sigma r)$$

$$\Delta_\Sigma v=e^{-\sigma r}(\Delta_\Sigma w-2\sigma\nabla_\Sigma w\cdot\nabla_\Sigma r+\sigma^2w\vert\nabla_\Sigma r\vert^2-\sigma w\Delta_\Sigma r)$$

From now on, all the integrals, without specifying the integral domain, are integrals over $\Sigma$. Integrate both sides by multiplying some functions, we have

$$\int re^{2\sigma r}(\Delta_\Sigma v)^2=\int r(\Delta_\Sigma w-2\sigma\nabla_\Sigma w\cdot\nabla_\Sigma r+\sigma^2w\vert\nabla_\Sigma r\vert^2-\sigma w\Delta_\Sigma r)^2$$

$$\geq -4\sigma\int r\nabla_\Sigma w\cdot\nabla_\Sigma r\Delta_\Sigma w-4\sigma\int r\nabla_\Sigma w\cdot\nabla_\Sigma rw(\sigma^2\vert\nabla_\Sigma r\vert^2-\sigma \Delta_\Sigma r)$$

Here we use the fundamental inequality $(a-b+c)^2\geq-4ab-4bc$. Integration by part, we have

$$=4\sigma\int \nabla_\Sigma(r\nabla_\Sigma w\cdot\nabla_\Sigma r)\cdot\nabla_\Sigma w+2\sigma\int w^2\di_\Sigma(r(\sigma^2\vert\nabla_\Sigma r\vert^2-\sigma\Delta_\Sigma r)\nabla_\Sigma r)=I+II$$

Now we are going to analyze each parts of the integrals. For the first integral,

$$I=4\sigma \int (\nabla_\Sigma w\cdot\nabla_\Sigma r)^2+4\sigma\int r\partial_j\partial_iw\partial_ir\partial_j w+4\sigma\int r\partial_iw\partial_jw\partial_i\partial_jr$$

Here we choose the local orthonormal basis, and Einstein notation to sum up all replicate indexes. $\partial$ should be understood as covariant derivative over $\Sigma$. Note since we choose the local orthonormal frame, we have $e_ie_j-e_je_i=(\nabla_{e_i}e_j)^\top=0$ on $\Sigma$, which means we can switch the position of $\partial_i$ and $\partial_j$. So we have

$$=4\sigma\int(\nabla_\Sigma w\cdot\nabla_\Sigma r)^2+2\sigma\int r\nabla_\Sigma(\vert\nabla_\Sigma w\vert^2)\cdot\nabla_\Sigma r+4\sigma\int r\partial_iw\partial_jw\partial_i\partial_jr$$

Integrate by parts again, we have

$$=4\sigma\int(\nabla_\Sigma w\cdot\nabla_\Sigma r)^2-2\sigma\int \vert\nabla_\Sigma w\vert^2(\vert\nabla_\Sigma r\vert^2+r\Delta_\Sigma r)+4\sigma\int r\partial_i w\partial_j w\partial_i\partial_j r$$

Plug in all the computation of differential or $r$ before the theorem, we have

$$=4\sigma\int(\nabla_\Sigma w\cdot x)^2\frac{1}{r^2}-2\sigma\int\vert\nabla_\Sigma w\vert^2(2+H\cdot x)+4\sigma\int r\partial_iw\partial_jw(\frac{\delta_{ij}}{r}+\frac{\langle x,\nabla_{e_j}e_i\rangle}{r}-\frac{x_ix_j}{r^3})$$

$$=-2\sigma\int\vert\nabla_\Sigma w\vert^2(H\cdot x)+4\sigma\int \partial_i w\partial_j w x\cdot A_{ij}$$

\begin{equation}
\geq -2\sigma\int\vert\nabla_\Sigma w\vert^2(H\cdot x+2\vert x\cdot A\vert)
\end{equation}

By $H\cdot x+2\vert x\cdot A\vert\geq 0$

$$\geq -2\sigma\int\vert\nabla_\Sigma w\vert^2(H\cdot x+2\vert x\cdot A\vert)=-2\sigma\int e^{2\sigma r}\vert v\sigma\nabla_\Sigma r+\nabla_\Sigma v\vert^2(H\cdot x+2\vert x\cdot A\vert)$$

$$\geq -4\sigma\int e^{2\sigma r}\sigma^2 \frac{\vert x^\top\vert^2}{r}v^2(H\cdot x+2\vert x\cdot A\vert)-4\sigma\int e^{2\sigma r}\vert\nabla_\Sigma v\vert^2(H\cdot x+2\vert x\cdot A\vert)$$

In order to throw away the gradient term, we need a reverse Poincare inequality as follows:

\begin{lem}
In the above setting, we have

\begin{equation}
\int e^{2\sigma r}\vert\nabla_\Sigma v\vert^2\leq \int e^{2\sigma r}v^2+\int e^{2\sigma r}(\Delta_\Sigma v)^2+4\sigma^2\int e^{2\sigma r}\frac{\vert x^\top\vert^2}{r^2}v^2
\end{equation}

\end{lem}

\begin{proof}[proof of lemma]

Integration by parts we have

$$\int e^{2\sigma r}v\Delta_\Sigma v=\int\nabla_\Sigma(e^{2\sigma r}v)\nabla_\Sigma v=\int e^{2\sigma r}\vert\nabla_\Sigma v\vert^2+2\sigma\int e^{2\sigma r}v\frac{x^\top}{r}\nabla_\Sigma v$$

So we have

$$\int e^{2\sigma r}\vert\nabla_\Sigma v\vert^2\leq \int e^{2\sigma r}\vert v\vert\vert\Delta_\Sigma v\vert+2\sigma\int e^{2\sigma r}\frac{\vert x^\top\vert}{r}\vert v\vert\vert\nabla_\Sigma v\vert$$

By Cauchy-Schwartz inequality, we have the first on RHS

$$\int e^{2\sigma r}\vert v\vert\vert\Delta_\Sigma v\vert\leq \frac{1}{2}\int e^{2\sigma r}v^2+\frac{1}{2}\int e^{2\sigma r}(\Delta_\Sigma v)^2$$

and the second term on RHS

$$2\sigma\int e^{2\sigma r}\frac{\vert x^\top\vert}{r}\vert v\vert\vert\nabla_\Sigma v\vert\leq 2\sigma^2\int e^{2\sigma r}\frac{\vert x^\top\vert^2}{r^2}v^2+\frac{1}{2}\int e^{2\sigma r}\vert\nabla_\Sigma v\vert^2$$

Then we get

$$\int e^{2\sigma r}\vert\nabla_\Sigma\vert^2\leq \int e^{2\sigma r}v^2+\int e^{2\sigma r}(\Delta_\Sigma v)^2+4\sigma^2\int e^{2\sigma r}\frac{\vert x^\top\vert^2}{r^2}v^2$$

\end{proof}

Back to the proof of main inequality. Let $S=(H\cdot x+2\vert x\cdot A\vert)$, then we have

$$II\geq -4\sigma^3\int e^{2\sigma r}\frac{\vert x^\top\vert^2}{r}v^2S-4\sigma\int e^{2\sigma r}v^2S-4\sigma\int e^{2\sigma r}(\Delta_\Sigma v)^2S-16\sigma^3\int e^{2\sigma r}\frac{\vert x^\top\vert^2}{r^2}v^2S$$

\begin{equation}
II\geq -20\sigma^3\int e^{2\sigma r}\frac{\vert x^\top\vert^2}{r^2}v^2S-4\sigma\int e^{2\sigma r}v^2S-4\sigma\int e^{2\sigma r}(\Delta_\Sigma v)^2S
\end{equation}

For the second term $II$, we have

$$II=2\sigma\int w^2\di_\Sigma\left(r(\sigma^2\vert\nabla_\Sigma r\vert^2-\sigma\Delta_\Sigma r)\nabla_\Sigma r\right)=2\sigma\int w^2\di_\Sigma\left((\sigma^2\frac{\vert x^\top\vert^2}{r}-\sigma(2+H\cdot x-\frac{\vert x^\top\vert^2}{r^2}))\frac{x^\top}{r}\right)$$

Now we compute the divergence in a local orthonormal basis, we have

$$\di_\Sigma\left((\sigma^2\frac{\vert x^\top\vert^2}{r}-\sigma(2+H\cdot x-\frac{\vert x^\top\vert^2}{r^2}))\frac{x^\top}{r}\right)=\partial_i\left((\sigma^2\frac{\vert x^\top\vert^2}{r}-\sigma(2+H\cdot x-\frac{\vert x^\top\vert^2}{r^2}))\frac{\langle e_i,x\rangle}{r}\right)$$

$$=\left(\sigma^2(\frac{(2x^\top+T)_i}{r}-\frac{\vert x^\top\vert^2x_i}{r^3})-\sigma \nabla_{e_i}H\cdot x+\sigma (\frac{(2x^\top+T)_i}{r^2}-2\frac{\vert x^\top\vert^2x_i}{r^4})\right)\frac{x_i}{r}$$

$$+\left(\sigma^2\frac{\vert x^\top\vert^2}{r}-\sigma(2+H\cdot x-\frac{\vert x^\top\vert^2}{r^2})\right)(\frac{\langle\nabla_{e_i}e_i,x\rangle+\langle e_i,e_i\rangle}{r}-\frac{x_ix_i}{r^3})$$

$$=\sigma^2(2\frac{\vert x^\top\vert^2}{r^2}+\frac{T\cdot x^\top}{r^2}-\frac{\vert x^\top\vert^4}{r^4})-\sigma\frac{1}{r}x_i\nabla_{e_i}H\cdot x+\sigma(2\frac{\vert x^\top\vert^2}{r^3}+\frac{T\cdot x^\top}{r^3}-2\frac{\vert x^\top\vert^4}{r^5})$$

$$+(\sigma^2\frac{\vert x^\top\vert^2}{r}-\sigma(2+H\cdot x-\frac{\vert x^\top\vert^2}{r^2}))(\frac{2+H\cdot x}{r}-\frac{\vert x^\top\vert^2}{r^3})$$

So we can write the integral of the second term as

$$II=2\sigma\int w^2(\sigma^2 A+\sigma B)$$

Where

$$A=2\frac{\vert x^\top\vert^2}{r^2}+\frac{T\cdot x^\top}{r^2}-\frac{\vert x^\top\vert^4}{r^4}+\frac{\vert x^\top\vert^2(2+H\cdot x)}{r^2}-\frac{\vert x^\top\vert^4}{r^4}\geq \frac{\vert x^\top\vert^2(2+H\cdot x)}{r^2}-\frac{C_1\cdot \vert x^\top\vert^2}{r^2}$$

$$B=-\frac{1}{r}x_i\nabla_{e_i}H\cdot x+2\frac{\vert x^\top\vert^2}{r^3}+\frac{T\cdot x^\top}{r^3}-2\frac{\vert x^\top\vert^4}{r^5}-\frac{(2+H\cdot x)^2}{r}-2\frac{\vert x^\top\vert^2(2+H\cdot x)}{r^3}+\frac{\vert x^\top\vert^4}{r^5}$$

$$\geq -\frac{1}{r}x_i\nabla_{e_i}H\cdot x+\frac{T\cdot x^\top}{r^3}-\frac{(2+H\cdot x)^2}{r}-2\frac{\vert x^\top\vert^2(2+H\cdot x)}{r^3}\geq -C(C_1,C_2,R_0)$$

Combine $I$ and $II$ together, and $\Sigma$ tilts no more than $\lambda$, we get the following inequality:

$$
(1+\sigma)C(C_1)\int re^{2\sigma r}(\Delta_\Sigma v)^2\geq(4-60C_1)(1-\lambda)^2\sigma^3\int e^{2\sigma r}v^2-C(C_1,C_2,R_0)\sigma^2\int e^{2\sigma r}v^2
$$

Then for $\sigma$ large enough, the second term on RHS will be dominant by the first term on RHS. So we finally get the inequality

\begin{equation}
\int re^{2\sigma r}(\Delta_\Sigma v)^2\geq C(C_1,C_2,R_0,\lambda)\sigma^2\int e^{2\sigma r}v^2
\end{equation}
\end{proof}

\section{Unique Continuation}

In this section we show our main Theorem \ref{thmu}. In particular, if $u$ is a harmonic function over $\Sigma$, $u$ satisfies unique continuation at infinity.

\begin{proof}[proof of Theorem \ref{thmu}]
Let us choose a cut-off function in $\mathbb{R}^n$. Suppose $\psi=\psi(r)$ is a radial symmetric cut-off function, supported on $(B_{R+1}\setminus B_{R_0+1})$, which is constant $1$ in $(B_{R}\setminus B_{R_0+2})$. Then restrict $\psi$ to $\Sigma$, still denote it by $\psi$, is a cut-off function on $\Sigma$. Note that $\vert\nabla_\Sigma\psi\vert\leq\vert\nabla\psi\vert$ and $\vert\Delta_\Sigma\psi\vert=\vert \di_\Sigma(\nabla_\Sigma v)\vert=\vert \di_\Sigma(v'\nabla_\Sigma r)\vert=\vert v''\vert\nabla_\Sigma r\vert^2+v'\Delta_\Sigma r\vert$ are all bounded by a constant $C_4$ independent of $R$. Here we use the assumption $\vert A\cdot x\vert$ is bounded by a constant $C_4$.

We apply the Carleman estimate Theorem \ref{thm1} to $\psi u$:

$$\int re^{2\sigma r}(\Delta_\Sigma (\psi u))^2\geq C\sigma^2\int e^{2\sigma r}(\psi u)^2$$

LHS is

$$\int re^{2\sigma r}(u\Delta_\Sigma\psi+2\nabla_\Sigma u\cdot\nabla_\Sigma\psi+\psi\Delta_\Sigma u)^2$$

$$\leq3\int re^{2\sigma r}u^2(\Delta_\Sigma \psi)^2+6\int re^{2\sigma r}\vert\nabla_\Sigma u\vert^2\vert\nabla_\Sigma\psi\vert^2+3\int e^{2\sigma r}\psi^2C_3^2(u^2)$$

Note for $\sigma$ large enough, $3C_3^2$ will be dominant by $\sigma^2/2$. We use $A_{s,t}$ to denote the annulus and we get the inequality:

$$C\sigma^2\int e^{2\sigma r}(\psi u)^2\leq\int_{A_{R_0+1,R_0+2}} re^{2\sigma r}u^2+\int_{A_{R,R+1}}re^{2\sigma r}u^2+\int_{A_{R_0+1,R_0+2}}re^{2\sigma r}\vert\nabla_\Sigma u\vert^2$$

$$+\int_{A_{R,R+1}}re^{2\sigma r}\vert\nabla_\Sigma u\vert^2$$

We next estimate the gradient term. Let $\eta=\eta(r)$ is another cut off function, $1$ on $A_{R,R+1}$ and $0$ outside $A_{R-1,R+2}$, then we have

$$\int_{A_{R,R+1}}re^{2\sigma r}\vert\nabla_\Sigma u\vert^2\leq\int re^{2\sigma r}\vert\nabla_\Sigma (\eta u)\vert^2$$

And by analogy of the reverse-Poincare inequality in previous section, we have

$$\int re^{2\sigma r}\vert\nabla_\Sigma (\eta^2 u)\vert^2\leq C\int_{A_{R-1,R+2}}(1+r+r^2+r^2\sigma^2)e^{2\sigma r}u^2\leq C\int_{A_{R-1,R+2}}e^{4\sigma r}u^2$$

So for $\sigma$ large enough, we have the following estimate:

$$\sigma^2\int_{A_{R_0+2,R}} e^{2\sigma r}u^2\leq C\left( \int_{A_{R_0+1,R_0+2}}re^{2\sigma r}u^2+\int_{A_{R,R+1}}re^{2\sigma r}u^2
+2\int_{A_{R_0+1,R_0+2}}re^{2\sigma r}\vert\nabla_\Sigma u\vert^2\right.
$$

$$+\left.\int_{A_{R-1,R+2}}e^{4\sigma r}u^2\right)$$

We have the following estimate for each terms on RHS:

$$\int_{A_{R_0+1,R_0+2}}re^{2\sigma r}u^2\leq (R_0+2)e^{\sigma(R_0+2)}\int_{A_{R_0+1,R_0+2}}u^2\leq Ce^{\sigma(R_0+2)}$$

$$\int_{A_{R_0+1,R_0+2}}re^{2\sigma r}\vert\nabla_\Sigma u\vert^2\leq (R_0+2)e^{\sigma(R_0+2)}\int_{A_{R_0+1,R_0+2}}\vert\nabla_\Sigma u\vert^2\leq Ce^{\sigma(R_0+2)}$$

$$\int_{A_{R,R+1}}re^{2\sigma r}u^2\leq e^{(2\sigma+1)(R+1)}\int_{A_{R,R+1}u^2}\leq C_5 e^{(2\sigma+1+t)(R+1)}\sup_{A_{R,R+1}}\vert u\vert^2$$

$$\int_{A_{R-1,R+2}}e^{4\sigma r}u^2\leq C_5e^{(4\sigma+t)(R+1)}\sup_{A_{R-1,R+2}}\vert u\vert^2$$

Combine all the terms above, we have the following estimate:

$$\sigma^2\int_{A_{R_0+2,R}}u^2\leq C+Ce^{(p\sigma+q)(R+1)}\sup_{A_{R,R+1}}\vert u\vert^2+Ce^{(p'\sigma+q')(R+1)}\sup_{A_{R-1,R+2}}\vert u\vert^2$$

For any $\epsilon>0$, let $\sigma$ so large that $C/\sigma^2\leq\epsilon$. then fixed this $\sigma$, let $R\to\infty$, we have

$$\int_{A_{R_0+2,\infty}}u^2\leq\epsilon$$

Then let $\epsilon\to0$ we have $u\equiv0$ on $\Sigma\setminus B_{R_0+2}$. Then by classical unique continuation of the solutions to elliptic equations with varying coefficients in Euclidean space (for example see ~\cite{K}), $u\equiv0$ on whole part of $\Sigma$ where it is defined.
\end{proof}

{\it Remark:} If we allow more decay of the function $u$, we can get unique continuation of solutions to some other elliptic inequality.

\section{Applications}
In this section we discuss some applications of the unique continuation theorem we obtained in the previous section.

\subsection{Unique Continuation at Infinity}

Note two dimensional subspaces $V=\mathbb{R}^2$ of $\mathbb{R}^n$ satisfying all the assumptions of unique continuation theorem in previous section, and for any such $\mathbb{R}^2$, its induced distance $r$ is exactly the intrinsic distance. Hence we have the following unique continuation theorem for $\mathbb{R}^2$:

\begin{thm}
Let smooth function $u$ defined over $\mathbb{R}^2\setminus B_{R_0}$ for some $R_0>0$ satisfying the elliptic inequality

$$\vert\Delta u\vert\leq \frac{C}{r^{1/2}}\vert u\vert$$

Then if $\limsup_{r\to\infty}\vert e^{\sigma r}u\vert=0$ for any $\sigma>0$, then $u\equiv 0$.

\end{thm}

{\it Remark:} By considering the function $\exp(-r\log(\log(\cdots\log(r))))$, we can see that for $u$ just satisfies $\vert\Delta u\vert\leq C\vert\log(\log(\cdots\log(r)))^2\vert\vert u\vert$ for $r$ large, $u$ may not satisfy unique continuation. The author conjecture that the unique continuation holds for solution $u$ to the elliptic inequality $\vert\Delta u\vert\leq C\vert u\vert$. In particular, all critical cases $u=e^{-\sigma r}$ satisfies this elliptic inequality  (where $C$ depends on $\sigma$).

If we just consider the case of harmonic functions, the exponentially decay requirement is not necessary. The polynomially decay, i.e. $\limsup_{r\to\infty}\vert r^\sigma u\vert=0$ leads to $u\equiv0$ in $\mathbb{R}^2$. The author conjecture that similar results hold for immersed surface case.

Also note that for a surface $\Sigma$ which is asymptotic to $2$-dim subspace of $\mathbb{R}^n$ at infinity, $\Sigma$ should satisfies the assumption of Theorem \ref{thmu}, with at most $r^2$ growth of the area. Hence the harmonic functions over such $\Sigma$ should also satisfies unique continuation. In fact, we can deal with surfaces with more ends.

\begin{thm}
Suppose $\Sigma$ is an immersed $2$-dim complete surface in $\mathbb{R}^n$, with finitely many ends. If every ends of $\Sigma$ are asymptotic to a plane or a cone, in the sense that it can be viewed as a graph of some function over the plane or the cone, and the function decays to $0$ as $r\to\infty$ in $C^3$ sense. Let $d$ be the intrinsic distance function on $\Sigma$ from a fixed point on $\Sigma$. Suppose function $u$ on $\sigma$ satisfies the elliptic inequality

$$\vert\Delta_\Sigma u\vert\leq\frac{C}{d^{1/2}}\vert u\vert$$

with exponentially intrinsic decay on $\Sigma$ at infinity, i.e.

$$\limsup_{d\to\infty}\vert e^{\sigma d}u\vert\to0\mbox{ for any $\sigma>0$}$$

then $u$ must be constant $0$.
\end{thm}

\begin{proof}
Since every ends of $\Sigma$ are asymptotic to a plane or a cone, we know that the assumptions in Theorem \ref{thmu} hold. The only thing left is to show that the intrinsic distance on $\Sigma$ is comparative to extrinsic distance on $\Sigma$, i.e. there is a constant $C$ such that

$$\frac{1}{C}r\leq d\leq Cr$$

This follows from the condition that the ends of $\Sigma$ are asymptotic to planes.
\end{proof}

From this result we can get an obstruction of surface embedded in Euclidean space asymptotically to planes, which is just Theorem \ref{app1} we mentioned before.

\subsection{Rigidity of Minimal Surfaces}

Unique continuation can be applied to show the rigidity of special geometric objects. If the geometric object has ends very similar to a plane or a cone, then it might coincident with the plane or the cone.

We first deal with the case of minimal surface in Euclidean space.

\begin{thm}\label{thmplane}
Suppose $\Sigma$ is an immersed minimal surface in $\mathbb{R}^n$ with exponential area growth, tilts no more than $\lambda<1$ and $\vert A\vert\leq C/r$. Then if $x_j$ is the coordinate function in $\mathbb{R}^n$ restricted to $\Sigma$ (note in previous sections $x_j$ always refers to the coordinates project to the orthonormal basis direction on $\Sigma$), with exponential decay:

$$\limsup_{r\to\infty}\vert e^{\sigma r}x_i\vert=0,\mbox{ for all $\sigma>0$}$$

Then $x_i\equiv0$.

In particular, if $\Sigma$ has finitely many ends, and one of the ends be asymptotic to a plane exponentially, in the sense that it can be viewed as a graph of a vector valued function over the plane, such that this function satisfies the exponential decay, then $\Sigma$ is just the plane.
\end{thm}

\begin{proof}
Minimality indicates $\nabla_\Sigma H\equiv0$, so it is easy to check $\Sigma$ satisfies the conditions in Theorem \ref{thmu}. For any surface in $\mathbb{R}^n$, we have

$$\Delta_\Sigma x_j=\langle\partial_j,H\rangle$$

Here $\partial_j$ is the direction of the coordinate in $\mathbb{R}^n$. Since $\Sigma$ is minimal, $\Delta_\Sigma x_j=0$, so by unique continuation we have $x_j\equiv0$ on $\Sigma$.

In particular, if one end of $\Sigma$ is asymptotically to a plane exponentially, we can choose a large ball $B_R$ and only consider the part of the end outside this ball, still denote it by $\Sigma$. Since it is asymptotically to a plane exponentially and $C^3$, it satisfies the assumptions in previous sections. Without loss of generality, we can assume the plane is $\{x_3=x_4=\cdots=x_n\}$. Then the minimal surface is $(x_1,x_2,x_3(x_1,x_2),\cdots,x_n(x_1,x_2))$. Now we apply the discussion in previous paragraphs to $x_3,x_4,\cdots,x_n$, they are become constant $0$. By the maximal principle $\Sigma$ is just the plane.
\end{proof}

One very special application is that

\begin{cor}
Let $\Sigma$ be a complete special Lagrangian surface in $\mathbb{R}^4$, for which one of its ends is asymptotic to a plane exponentially, then it coincides with the plane.
\end{cor}

Next we restrict our discussion to surfaces in $\mathbb{R}^3$. In this case, plane is a very special case of minimal surfaces. We will develop a rigidity result for more general minimal surfaces. Here we state some results of general minimal graphs in \cite{CM} Chapter 2 section 6.

Let $\Sigma$ be any surface in $\mathbb{R}^3$, $N(x)$ be the normal vector at $x$ of $\Sigma$. Then we can define $\Sigma_u$ the graph of function $u$ over $\Sigma$ by $$\Sigma_u=\{x+u(x)N(x)\vert x\in \Sigma\}$$ Then we have following lemma for minimal graph:

\begin{lem}{(Lemma 2.26 in \cite{CM} Chapter 2 section 6)}
There is a constant $C$ so that if $\Sigma$ is a minimal surface, the graph of $u$ over $\Sigma$ is also minimal, and $\max{\vert u\vert\vert A\vert,\vert\nabla u\vert}\leq1$, then $u$satisfies the equation
\begin{equation}
\di[(I+\bar L)\nabla u]+(1+Q)\vert A\vert^2u+Q_{ij}A_{ij}=0
\end{equation}
where $\bar L\leq C(\vert u\vert\vert A\vert+\vert\nabla u\vert),\vert Q\vert\leq C(\vert u\vert\vert A\vert+\vert\nabla u\vert)^2$ and $\vert Q_{ij}\vert\leq C(\vert u\vert\vert A\vert+\vert\nabla u\vert)^2$
\end{lem}

\begin{cor}
Under the assumption in previous lemma, $u$ satisfies the following equation
\begin{equation}
\Delta u+\vert A\vert^2u=\tilde Q(\vert u\vert\vert A\vert,\nabla u)
\end{equation}
where $\tilde Q(x,y)\leq C(\vert x\vert^2+\vert y\vert^2)$.
\end{cor}

Now we get an equation for general minimal graph. Apply unique continuation theorem we get the following rigidity of ends of general minimal surface:

\begin{thm}\label{thmr}
Suppose $\Sigma$ is a minimal surface in $\mathbb{R}^3$ with at most exponential area growth and tilts no more than $\lambda<1$. Moreover we assume $\vert A\vert\leq C/r^{1/4}$ for some constant $C$. Then if $\Sigma'$ is a minimal surface, which asymptotic to a same end with $\Sigma$ in the following sense: $\Sigma'$ can be viewed as a graph $\Sigma_u$ of function $u$ over $\Sigma$, and for all $\sigma >0$ $$\limsup_{r\to\infty}\vert e^{\sigma r}u\vert=0,\vert\nabla u\vert \leq C\frac{\vert u\vert^{1/2}}{r^{1/4}}$$
Then $\Sigma'$ must coincide with $\Sigma$.
\end{thm}

\begin{proof}
$\Sigma$ satisfies the conditions in Theorem \ref{thmu}, and under the assumptions of asymptotic behavior here $u$ satisfies a elliptic inequality as in \ref{thmu}. So we can apply unique continuation theorem to conclude the result.
\end{proof}

In other word, if two minimal surfaces have one of each ends be asymptotic to same end under certain meaning, then they must coincident.

There are many results of ends behavior of minimal surface in $\mathbb{R}^3$, for example in \cite{MP} Meeks and P\'erez list many classical results. In particular, there are a lot of minimal surfaces in $\mathbb{R}^3$ with ends asymptotic to same thing. For example in \cite{WW}, Weber and Wolf construct a series of embedded minimal surface with two catenoidal ends and many planar ends; well-known Sherk' singly periodic surface has two planar ends. So the rigidity result only holds under stronger meaning of asymptoticity, like what we assume in our theorem.

Our assumption in Theorem \ref{thmr} holds for many minimal surfaces. For example:

\begin{itemize}
    \item Plane
    \item Catenoid
    \item Helicoid
    \item Enneper Surface (See \cite{Enneper} for detailed information)

\end{itemize}

Thus we have the following corollary:

\begin{cor}
Any minimal surface has one end asymptotic to one end of the minimal surfaces in the above list, in the sense in Theorem \ref{thmr}, must coincide with the whole minimal surface.
\end{cor}

\subsubsection{Rigitidy of Self Shrinker}
Another interesting surfaces appearing in geometric analysis are self-shrinkers. They are homothetic solutions to mean curvature flow at time $-1$. It satisfies the equation

\begin{equation}
H=\frac{x^\bot}{2}
\end{equation}

We obtain the following rigidity theorem for self-shrinkers:

\begin{thm}
Let $\Sigma$ be a $2$-dim complete self-shrinker in $\mathbb{R}^n$, one end of it is asymptotic to a plane exponentially, which tilts no more than $C/r^{1/2}$ for some fixed constant $C$. Then $\Sigma$ coincides with the plane.
\end{thm}

\begin{proof}
In this case, we have

$$\Delta_\Sigma x_j=\frac{(x^\bot)_j}{2}\leq\frac{C\vert x_j\vert}{2r^{1/2}}$$

then apply Theorem \ref{thmu} just like in previous theorem, we get the result.
\end{proof}

{\it Remark}: Besides the rigidity result for minimal surfaces in $\mathbb{R}^3$, the rest results of this section can be applied to surfaces with arbitrary codimensions.

\end{document}